\documentclass[a4paper,10pt]{amsart}
\usepackage{amsfonts,amsthm,amsmath,amssymb}
\usepackage{graphicx}
\usepackage[utf8]{inputenc}

\usepackage{hyperref}

\usepackage{todonotes}
\usepackage{xypic}

\newcommand{\FF}{{\mathbb{F}}}

\newcommand{\KK}{{\mathbb{K}}}
\newcommand{\ZZ}{{\mathbb{Z}}}
\newcommand{\hG}{\widehat{G}}

\newcommand{\cA}{{\mathcal{A}}}
\newcommand{\cF}{{\mathcal{F}}}
\newcommand{\cG}{{\mathcal{G}}}
\newcommand{\cK}{{\mathcal{K}}}
\newcommand{\cM}{{\mathcal{M}}}
\renewcommand{\cL}{{\mathcal{L}}}

\newcommand{\Ker}{\mathop{\mathrm{Ker}}}
\renewcommand{\Im}{\mathop{\mathrm{Im}}}

\newcommand{\mN}{{\mathcal{N}}}
\newcommand{\Id}{{\mathop{\mathrm{Id}}}}

\theoremstyle{plain}
\newtheorem{theorem}{Theorem}[section]
\newtheorem{proposition}[theorem]{Proposition}
\newtheorem{lemma}[theorem]{Lemma}

\newtheorem{problem}[theorem]{Problem}

\theoremstyle{remark}

\newtheorem{remark}[theorem]{Remark}
\newtheorem{question}[theorem]{Question}

\theoremstyle{definition}
\newtheorem{definition}[theorem]{Definition}
\newtheorem{example}[theorem]{Example}

\title{Ribbon graphs and bialgebra of Lagrangian subspaces}
\author{Victor Kleptsyn}
\address{CNRS, Institut de Recherches Math\'ematiques de Rennes (UMR 6625 du CNRS)\\
Campus Beaulieu, 35042 Rennes, France
}
\email{victor.kleptsyn@univ-rennes1.fr}
\author{Evgeny Smirnov}
\address{
Department of Mathematics and AG Laboratory\\ National Research University  Higher School of Economics, Russian Federation\\
Vavilova St. 7, 112312 Moscow, Russia}
\address{Independent University of Moscow, Bolshoi Vlassievskii per., 11, 119002 Moscow, Russia}
\email{esmirnov@hse.ru}
\date{\today}
\subjclass[2010]{57M20, 57M27 (primary), 57M15, 05C25 (secondary)}
\keywords{Ribbon graph, chord diagram, Vassiliev knot invariant}
\dedicatory{Dedicated to the memory of Sergei Vassilievich Duzhin}

\begin{document}

\maketitle

\begin{abstract}
To each ribbon graph we assign a so-called $L$-space, which is a Lagrangian subspace in an even-dimensional vector space with the standard symplectic form. This invariant generalizes the notion of the intersection matrix of a chord diagram. Moreover, the actions of Morse perestroikas (or taking a partial dual) and Vassiliev moves on ribbon graphs are reinterpreted nicely in the language of $L$-spaces, becoming changes of bases in this vector space. Finally, we define a bialgebra structure on the span of $L$-spaces, which is analogous to the 4-bialgebra structure on chord diagrams.
\end{abstract}

\tableofcontents

\section{Introduction}

Chord diagrams appear as an essential tool in the study of Vassiliev finite-type knot invariants. They are constructed as follows: to each singular knot, i.e., a knot with a finite number of simple self-intersection points, we assign a circle, corresponding to the preimage of the knot, with the preimages of the self-intersection points joined by chords.

Chord diagrams span a graded vector space (over an arbitrary field) with grading corresponding to the number of chords. Since the finite-type knot invariants produce functions invariant under 4-term relations, they can be considered as functions on the quotient of this vector space by the subspace generated by the 4-term relations.  The latter vector space can be equipped with operations of multiplication and comultiplication, turning it into a bialgebra.

The combinatorial data encoded in chord diagrams are quite involved. In some situations one can pass from chord diagrams to their intersection matrices: these are symmetric 0-1 matrices with rows and columns corresponding to the chords, such that each element is 1 if the corresponding two chords intersect and 0 otherwise (all the diagonal entries are 0). Alternatively, instead of matrices one can consider their intersection graphs. As in the case of chord diagrams, we can introduce the four-term relation and bialgebra structure on the algebra spanned by intersection matrices. The correspondence between chord diagrams and their intersection matrices is neither injective nor surjective; however, weight systems on intersection matrices can be lifted to a weight system on chord diagrams,
so intersection matrices can serve a valuable source of new weight systems.

The notion of chord diagram admits several natural generalizations. First, we can study (singular) links instead of knots; they give rise to  multicomponent chord diagrams. Further, following V.~Arnold~\cite{Arnold93}, instead of knots we can consider smooth curves in the plane with possible transversal self-intersections. In this case chord diagrams are replaced by \emph{framed} chord diagrams, with chords of two different types. These chords can be thought of as ribbons joining pairs of segments on a circle; these ribbons can be either untwisted or twisted. If all ribbons are untwisted, the diagram is unframed, i.e., ordinary. 

For framed chord diagrams one can also introduce the four-term relation; this was done by S.~Lando in \cite{Lando06}, cf. also \cite{LZ05}. However, the space of framed chord diagrams modulo the four-term relation does not form a bialgebra anymore. It admits a coalgebra structure, but the multiplication of framed chord diagrams is not well-defined, cf.~\cite{IlyutkoManturov15}. Recently M.~Karev~\cite{Karev14} showed that this space is a module over the algebra of unframed chord diagrams.

However, for a  framed diagram we still can define the notion of its intersection matrix (or, equivalently, of its intersection \emph{framed graph}, i.e., a graph with colored vertices of two different colors, corresponding to twisted and untwisted ribbons). The four-term relation descends to the space of intersection graphs. In \cite{Lando00} it was shown that the space of framed graphs modulo the four-term relation still forms a bialgebra (while the space of framed chord diagrams does not).

Both unframed and framed chord diagrams can be considered in a much more general context: they are  \emph{ribbon graphs}, or \emph{fat graphs}, with one vertex. Informally speaking, ribbon graphs are surfaces glued from disks representing vertices and ribbons representing edges (formal definition will be given in Sec.~\ref{sec:ribbons}). Thus, one can think of ribbon graphs with more than one vertex as of \emph{framed multicomponent chord diagrams}.

In this paper we introduce a new object, called the \emph{$L$-space}, for each ribbon graph. For a ribbon graph with $n$ edges, this is a Lagrangian subspace in a $2n$-dimensional space over $\FF_2$ with a standard symplectic form.  It generalizes the notion of the intersection matrix. A vector space spanned by $L$-spaces has a bialgebra structure, which extends the bialgebra structure on the space spanned by the set of (framed) intersection matrices.

The notions  of Vassiliev moves and four-term relations for chord diagrams can be generalized to ribbon graphs. We show that these notions are reinterpreted nicely in terms of $L$-spaces: they correspond to certain base changes in the ambient space. Moreover, for ribbon graphs there is an operation of \emph{partial dual}, or a \emph{Morse perestroika} with respect to a subset of chords, defined by Chmutov \cite{Chmutov09}; analogues of this operation also appeared before in the study of graph invariants (see \cite{ABS02a} and \cite{ABS02b}). We show that this operation also descends to $L$-spaces as a particularly nice base change.

Finally, we define the bialgebra of $L$-spaces purely combinatorially. This notion generalizes  the notion of the bialgebra of  graphs, introduced by S.~Lando in~\cite{Lando00}. More specifically, we define a bialgebra structure on the vector space freely spanned by all possible Lagrangian spaces. This bialgebra contains the bialgebra of framed graphs. The Vassiliev moves, introduced earlier, allow us to extend the four-term relation from the bialgebra of framed graphs to this larger bialgebra. The set of four-term relations generates an ideal; the quotient bialgebra is a natural analogue of Lando's bialgebra of graphs. An interesting question is to find the dimensions of its graded components.

\subsection*{Plan of the paper} This paper is organized as follows. In Section~\ref{sec:ribbons} we recall the definition of  ribbon graphs and define the operations of taking a partial dual and Vassiliev moves. In Section~\ref{sec:l-space} we provide a construction of $L$-space of a ribbon graph and discuss its relation with the notion of intersection graph of a chord diagram. Subsection~\ref{ssec:vassiliev} we provide explicit formulas for the action of partial duals and Vassiliev moves of ribbon graphs on the corresponding $L$-spaces. In Section~\ref{sec:bialgebra} we prove our main result: we describe a bialgebra structure on the vector space spanned by $L$-spaces, and discuss its relation to the bialgebras of chord diagrams and graphs. Finally, in Section~\ref{sec:misc} we  provide explicit formulas for the action of partial duals on chord diagrams point out a relation of our operations on ribbon graphs with certain graph invariants defined by Arratia, Bollob\'as, and Sorkin, and conclude with discussing how our construction could be related to the study of multicomponent plane curves. 

\section{Ribbon graphs}\label{sec:ribbons}

\subsection{Definition of ribbon graphs} 

We begin with a formal definition, taken from \cite{Chmutov09} and going back to \cite{BollobasRiordan02}.

\begin{definition} A \emph{ribbon graph} $G$ is a surface (possibly non-orientable) with boundary, represented as the union of two sets of closed topological disks called \emph{vertices} $V(G)$ and \emph{edges} $E(G)$, satisfying the following conditions:
\begin{itemize}
\item edges and vertices intersect by disjoint line segments;

\item each such segment lies in the closure of precisely one edge and one vertex;

\item each edge contains two such segments.
\end{itemize}
\end{definition}

We consider ribbon graphs up to a homeomorphism preserving the decomposition into vertices and edges. 

Informally speaking, vertices can be thought of as disks, and edges are rectangles attached to the boundary of these disks by two opposite sides; see Fig.~\ref{fig:ribbongraphs} below.

\begin{figure}[h!]
\center
\includegraphics{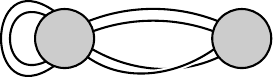}
\caption{Example of a ribbon graph}\label{fig:ribbongraphs}
\end{figure}

Note that an edge can join a vertex with itself; such an edge will be called a \emph{loop}. A loop with a vertex attached to it can be homeomorphic either to a cylinder or to a M\"obius band; it is then called \emph{orientable} or \emph{disoriented}, or \emph{twisted}. Note that for an edge joining two different vertices there is no natural notion of orientation (however we can say whether two such edges have the same orientation or not).

Alternatively, a ribbon graph can be viewed as a usual graph with some additional data: for each vertex we fix a cyclic order of half-edges adjacent to this vertex, and for each cycle in this graph we fix its orientation, i.e. we say whether this cycle is orientable (a cylinder) or disorienting (a M\"obius band), with natural conditions of coherence of orientation for cycles with common edges.

Recall the definition of a framed chord diagram (cf.~\cite{Lando06}).

\begin{definition} A \emph{chord diagram of order $n$} is a circle with $n$ pairs of pairwise distinct points, considered up to a diffeomorphism (not necessarily orientation-preserving), called \emph{chords}. \emph{Framing} is a map from the set of chords into the set $\FF_2=\{0,1\}$; the chords mapped into 0 and 1 are called \emph{orientable} and \emph{disoriented}, respectively.
\end{definition}

We will represent chord diagrams by circles with marked points joined by arcs, solid for the oriented chords and dashed for the disorienting ones.

Chord diagrams (both framed and unframed ones) can be considered as ribbon graphs with one vertex. To pass from a chord diagram to the ribbon graph, we need to consider the circle as the boundary of a unique vertex and to ``thicken'' each of its chords, replacing it by a ribbon attached to the boundary according to its framing; see Fig.~\ref{fig:chordribbon} below. The resulting surface is orientable if and only if all chords of the diagram are orientable (i.e., have framing 0).

\begin{figure}[h!]
\center
\includegraphics{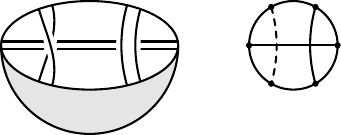}
\caption{Chord diagram and the corresponding ribbon graph}\label{fig:chordribbon}
\end{figure}

\subsection{Vassiliev moves for ribbon graphs}

In this subsection we introduce two families of involutive operations on ribbon graphs, called \emph{first} and \emph{second Vassiliev moves}. These operations generalize Vassiliev moves for chord diagrams, used in the definition of the four-term relation (see below).

Let $G$ be a ribbon graph, and let $v$ be its vertex of degree at least 2. Consider two edges $i$ and $j$ attached to $v$ in such a way that there are no other edges attached to the same vertex between them\footnote{Strictly speaking, the Vassiliev moves depend not only on $i$ and $j$, but also on choosing a half-edge on each of them.}; i.e., the corresponding half-edges are neighboring for the cyclic order on $v$. Let us define the \emph{Vassiliev moves} for $G$ as follows.

The result of the first move $v_1^{ij}(G)$ is obtained from $G$ by attaching $i$ and $j$ to $v$ along the same segments, but in the different order, as shown on Fig.~\ref{fig:ribbon_moves}. The second Vassiliev move acts on $G$ as follows: let us take the second edge $j$ and slide it along the edge $i$ until it reaches the vertex on its second endpoint of it (see Fig.~\ref{fig:ribbon_moves}). Note that $v_2$ does not affect the topology of a ribbon graph: as surfaces $G$ and $v^{ij}_2(G)$ are homeomorphic (while, of course, as ribbon graphs they are different).

While describing the second Vassiliev move we will call $i$ the \emph{fixed chord}, and $j$ will be referred to as the \emph{moving chord}. Note that the choice of moving and fixed chords is important: $v_2^{ij}\neq v_2^{ji}$.

\begin{figure}[h!]
\center
\includegraphics[width=2.5cm]{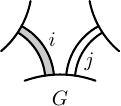} \quad  
\includegraphics[width=2.5cm]{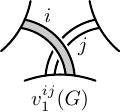} \quad
\includegraphics[width=2.5cm]{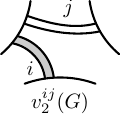} \quad 
\includegraphics[width=2.5cm]{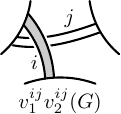}
\caption{Vassiliev moves for ribbon graphs}\label{fig:ribbon_moves}
\end{figure}

It is easy to see that the Vassiliev moves are involutive operations, and that they commute. They do not change the number of edges and vertices of $G$.

An important particular case is when $G$ has only one vertex, i.e. corresponds to a chord diagram. If in addition $G$ is oriented, we get the classical notion of Vassiliev moves for chord diagrams, used in the Vassiliev knot invariant theory, depicted on Figure~\ref{fig:chord_moves} below (see \cite{LZ05} or \cite{CDM} for details). On this figure we indicate only the fixed and the moving chord; the remaining part of the diagram is the same for all four chord diagrams, corresponding to $G$, $v_1^{ij}(G)$, $v_2^{ij}(G)$ and $v_1^{ij}v_2^{ij}(G)$.

\begin{figure}[h!]
\center
\includegraphics{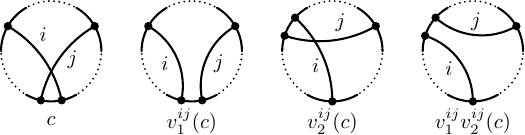}
\caption{Vassiliev moves for chord diagrams}\label{fig:chord_moves}
\end{figure}

In \cite{Lando06} Lando generalizes the notion of Vassiliev moves for the case of \emph{framed} chord diagrams, distinguishing between three types of Vassiliev moves for all possible framings of the fixed and a moving chord. In our terms framed chord diagrams correspond to ribbon graphs with one vertex, not necessarily oriented. It turns out that for one-vertex graphs our definition of Vassiliev moves coincides with the one from \cite{Lando06}.

\begin{proposition} Let $G$ be a ribbon graph with one vertex (possibly non-oriented). Then the Vassiliev moves applied to $G$ coincide with the Vassiliev moves for framed chord diagrams defined in~\cite[Fig.~2--4]{Lando06}.
\end{proposition}

The proof of this proposition is straightforward.

\subsection{Partial duals, relation with Vassiliev moves}\label{ssec:partialdual} In \cite{Chmutov09} S.~Chmutov gives a definition of the \emph{partial dual} of a ribbon graph $G$ with respect to an edge $i\in E(G)$. Here we recall this definition in a slightly modified way. 
\begin{definition}
Let $G$ be a ribbon graph, $i$ one of its ribbons. Take $S$ to be the union of borders of vertices (that are discs, and hence $S$ is a union of circles), and consider the symmetric difference $S':=S\mathop{\Delta} \partial (i)$: we remove from $S$ two arcs along which $i$ touches it, and add back two ``side boundary'' arcs from~$\partial(i)$. Then, $S'$ is still a union of circles (as closed one-dimensional manifold). Take a union of disc attached all the circles from $S'$ (these will be the new vertices), and glue back the same ribbons as in~$G$ (see Fig.~\ref{f:Morse}). The obtained graph is called \emph{partial dual} (or \emph{Morse petestroika}) of $G$ by the edge (or ribbon)~$i$.
\end{definition}
\begin{figure}[!h]
\center
\includegraphics{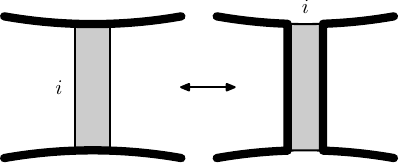}
\caption{Morse perestroika (or partial dual). Bold lines show the boundary of the vertices before and after taking the dual.}\label{f:Morse}
\end{figure}

We will denote the partial dual of $G$ with respect to an edge $i$ by $\mu_i(G)$. This notation is not standard: usually such a graph is denoted by $G^i$. The following lemma is straightforward (cf. also~\cite[Lemma 1.8 (a)--(c)]{Chmutov09}):

\begin{lemma} \begin{itemize}
\item $\mu_i$ is involutive: $\mu_i^2(G)=G$;

\item For $i\neq j$, the corresponding partial duals commute: $\mu_i\circ\mu_j(G)=\mu_j\circ\mu_i(G)$. Hence, an operation of \emph{partial dual with respect to a set of edges} $E'=\{i_1,\dots,i_k\}$ is well-defined: $\mu_{E'}(G)=\mu_{i_1}\circ\dots\circ\mu_{i_k}(G)$;

\item For arbitrary subsets $E',E''\subset E(G)$, we have $\mu_{E'}\circ\mu_{E''}(G)=\mu_{E'\triangle E''}(G)$, where $E'\triangle E''$ stands for the symmetric difference of $E'$ and $E''$.
              \end{itemize}
\end{lemma}

So the operation of partial dual defines the action of the group $\ZZ_2^m$ on the set of ribbon graphs with $m$ edges.

A key relation between the partial duals and Vassiliev moves is as follows.

\begin{proposition}\label{prop:conjugacy} First and second Vassiliev moves are conjugate by the partial dual with respect to the fixed chord:
\[
 v_2^{ij}(G)=\mu_i \circ v_1^{ij}\circ\mu_i (G).
\]
\end{proposition}

\begin{proof}
Let $i$ and $j$ be two neighboring edges. First suppose that they are not loops (i.e., have distinct endpoints). The proof is then shown on Fig.~\ref{fig:conjugacy}. The cases when either $i$ or $j$ (or both) are loops are considered in a similar way.
\begin{figure}[h!]
\center
\includegraphics[width=7cm]{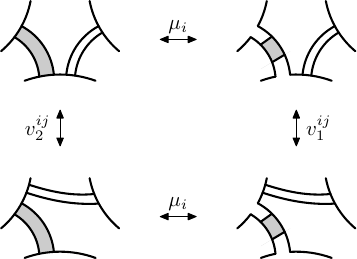}
\caption{Vassiliev moves are conjugate (the chord $i$ is shown by grey color, the chord $j$ is unfilled)}\label{fig:conjugacy}
\end{figure}
\end{proof}

\section{$L$-spaces of ribbon graphs}\label{sec:l-space}

In this section we introduce our fundamental object: the $L$-space of a ribbon graph. It is a generalization of the intersection graph of a chord diagram, introduced by S.~Chmutov, S.~Duzhin, and S.~Lando in \cite{ChmutovDuzhinLando94}. We start with recalling the latter notion.

\subsection{Intersection graph of a chord diagram}\label{ssec:IntGraph}
Let $c$ be a chord diagram with $n$ chords. Let us number these chords by $1,\dots,n$ in a certain way and consider a square matrix $A(c)$ of order $n$ with coefficients from $\FF_2$, defined as follows:
\begin{eqnarray*}
a_{ii}&=&0;\\
a_{ij}&=&\begin{cases} 1 &\text{if the chords $i$ and $j$ intersect;}\\
0 &\text{otherwise}.
\end{cases}
\end{eqnarray*}

Since the chords in a chord diagram are not numbered, the intersection matrix is also defined up to the  conjugation by a permutation matrix, i.e., up to a simultaneous permutation of rows and columns.

The intersection matrix of a chord diagram has a nice topological interpretation. The ribbon graph $G$ corresponding to a chord diagram with $n$ chords is a surface with boundary. The first homotopy group $H_1(G,\FF_2)$ of this surface is $n$-dimensional; it admits a natural basis formed by the chords (see~\cite[Sec~4.8.2]{CDM}). Then the intersection matrix for this basis in the homology coincides with the intersection matrix of the chord diagram, as defined above (see Fig.~\ref{fig:chord-base}, left).

 Equivalently, instead of the matrix we can consider the \emph{intersection graph} of a chord diagram. It is a graph whose vertices are indexed by the chords of a chord diagram, and each two vertices are connected by an edge iff the corresponding chords intersect each other.

The map from chord diagrams to their intersection graphs is obviously not injective: there exist different diagrams with the same graph. It is also non-surjective: Bouchet \cite{Bouchet} gives a criterion for a graph to be the intersection graph of a chord diagram. For instance, the simplest graph that does not correspond to any chord diagram is the ``5-wheel'' graph, shown on Fig.~\ref{f:5w2}.

\begin{figure}[!h]
\center
\includegraphics{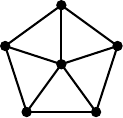}
\caption{``5-wheel'' graph}\label{f:5w2}
\end{figure}

However, the intersection graphs keep track of much information about chord diagrams (see \cite[4.8]{CDM} or \cite{LZ05}). Just like chord diagrams, their intersection graphs are subject of a four-term relation, they can be formed into a bialgebra and so forth; we postpone this discussion until Sec.~\ref{ssec:bialg_graphs}.

The notion of intersection matrices also makes sense for framed chord diagrams: if, as above, we define the intersection matrix of a diagram as the intersection matrix in the cohomology of the corresponding ribbon graph in the basis formed by the chords, its elements will admit the following simple description:
\begin{eqnarray*}
a_{ii}&=&\begin{cases}0 & \text{if the $i$-th chord is orientable;}\\
1 & \text{if the $i$-th chord is twisted;}
\end{cases}
\\
a_{ij}&=&\begin{cases} 1 &\text{if the chords $i$ and $j$ intersect;}\\
0 &\text{otherwise}.
\end{cases}
\end{eqnarray*}
(the self-intersection index of a twisted chord is 1). Likewise, we can define \emph{framed intersection graphs} of framed chord diagrams, with vertices of two different colors, corresponding to orientable and twisted chords. However, the intersection matrix (or graph) does not admit a direct generalization for the case of arbitrary ribbon graphs (or, what is the same, of multicomponent framed chord diagrams). Our next goal is to present a (what we think is a) reasonable
analogue of this notion.

\subsection{Definition of $L$-space}\label{ssec:def_lspace}

Let $G$ be a connected ribbon graph with $n$ chords (for the moment, we will suppose for simplicity that these chords 
are numbered). Consider a $2n$-dimensional symplectic space $V=\FF_2^{2n}$ with the base $e_1,\dots,e_n,f_1,\dots,f_n$, and the intersection form defined by $(e_i,f_j)=\delta_{ij}$ (and vanishing on any other couple of base vectors). Our goal is to associate to the graph $G$ a Lagrangian (i.e. maximal isotropic) subspace $L(G)\subset V$.

To do this, first consider the surface $\hG$ obtained from $G$ by removing an open disk from each of its vertices. We will call this surface \emph{punctured ribbon graph}; this is a compact surface with boundary. Next, for each edge $E_i$ of $G$ consider two paths: let $\alpha_i$ connect the opposite sides of the edge $E_i$ and let $\beta_i$ connect along $E_i$ the boundaries of the punctures in the vertices adjacent to $E_i$, as shown on Figure~\ref{fig:edge_cycles}.
\begin{figure}[h!]
\center
\includegraphics{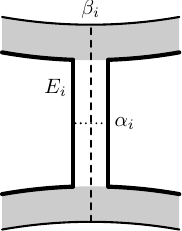}
\caption{Two cycles corresponding to an edge; here $\partial G$ is shown by bold lines, the boundary of punctures $\partial \widehat{G}\setminus \partial G$ by thin ones, (punctured) vertices are filled with grey color, the edge $E_i$ is unfilled.}\label{fig:edge_cycles}
\end{figure}

The paths $\alpha_i$ and $\beta_i$ define classes in the first relative homology group of $\widehat{G}$ modulo its boundary: $[\alpha_i],[\beta_i]\in H_1(\widehat{G}, \partial \widehat{G};\FF_2)$ (we suppose that all homology groups have coefficients in $\FF_2$).

Now, the intersection index defines a natural $\FF_2$-valued pairing between the relative homology group $H_1(\widehat{G}, \partial \widehat{G};\FF_2)$ and the first homology group $H_1(\widehat{G};\FF_2)$. Consider the intersection indices of a cycle $\gamma\in H_1(\widehat{G};\FF_2)$ with $[\alpha_i]$ and $[\beta_i]$. Roughly speaking, the intersection index with $[\alpha_i]$ tells us whether $\gamma$ ``crosses'' $E_i$ from one shore to the other, while the intersection index with $[\beta_i]$ tells us whether $\gamma$ goes along the boundary of the vertex from the left to the right near one of the endpoints of~$E_i$.

Let us group all these intersection indices together, defining a map
\begin{equation}\label{eq:map}
\varphi_G:H_1(\widehat{G};\FF_2)\to V, \quad \varphi_G(\gamma)=\sum_i (\gamma\cap [\alpha_i]) \cdot e_i + \sum_i (\gamma\cap [\beta_i]) \cdot f_i.
\end{equation}

This provides us with the following

\begin{definition}\label{def:lspace}
\emph{L-space}, associated to the ribbon graph $G$, is defined as the image of the map $\varphi_G$:
$$
L(G):=\Im \varphi_G \subset V.
$$
\end{definition}

One can immediately note that $L(G)$ is at most $n$-dimensional. Indeed, $\dim H_1(\widehat{G};\FF_2)=n+1$, as $\widehat{G}$ is homotopy equivalent to the wedge of $n+1$ circle. However, the ``outer'' boundary $\partial G\in H_1(\widehat{G};\FF_2)$ does not intersect any $\alpha_i$, and intersects any $\beta_i$ twice; hence, $\varphi_G(\partial G)=0$. 

The following example motivates this definition as a generalization of the notion of the intersection matrix:

\begin{example}\label{ex:int}
Let $G$ be a ribbon graph, corresponding to a chord diagram~$c$ with $n$ chords. Then, $L(G)\subset V$ in an $n$-dimensional space, generated by the rows of the matrix
\begin{equation}\label{eq:rows}
M=(\Id_n \mid	 A(c))
\end{equation}
where $A(c)$ is the intersection matrix of the chord diagram.
\end{example}

Indeed, the $i$-th row corresponds to the $i$-th chord: closing it ``along the boundary circle'', we get a cycle $\gamma_i\in H_1(\hG;\FF_2)$ that intersects $\alpha_i$ and that has the intersection index with $\beta_j$ equal to 1 if and only if $i$-th and $j$-th chords intersect (see Fig.~\ref{fig:chord-base}, right).
\begin{figure}[h!]
\center
\includegraphics{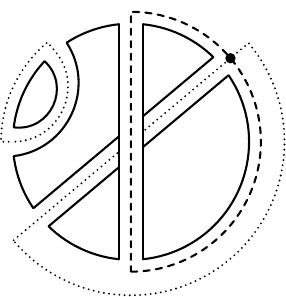} \qquad
\includegraphics{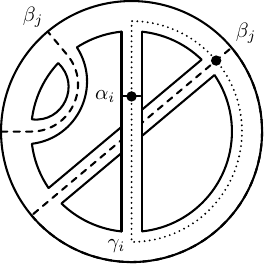}
\caption{Left: two cycles $\gamma_i$ and $\gamma_j$ intersect if and only if the corresponding chords intersect. 
Right: cycle $\gamma_i$ and its intersections with $\alpha_j$, $\beta_j$.}\label{fig:chord-base}
\end{figure}

\begin{remark}\label{r:chord-intersection}
Conclusion of Example~\ref{ex:int} still holds if one passes to the framed chord diagrams.
\end{remark}

In the beginning of this section, we have claimed that the space $L(G)$ will be 
Lagrangian. Let us show that it is indeed the case.
\begin{theorem}\label{thm:Lagrange}
Let $G$ be a ribbon graph. Then, $L(G)$ is a Lagrangian (i.e. maximal isotropic) subspace of~$V$.
\end{theorem}

Before passing to its proof in full generality, we make one simple remark in the simplest particular case: 
\begin{remark}
For the particular case of a one-vertex ribbon graph, corresponding to a framed chord diagram~$c$, 
the pairing between the $i$-th and $j$-th rows $r_i,r_j\in V$ of the matrix~$M$ given 
in~\eqref{eq:rows} vanishes for all $i$, $j$. Indeed, it is equal to the 
$$
(r_i, r_j) = a_{ij}+a_{ji} =0\in \FF_2,
$$
where $a_{ij}$ are the elements of the intersection matrix~$A(c)$, and the 
sum $a_{ij}+a_{ji}$ vanishes in $\FF_2$, as this matrix is symmetric.
\end{remark}

\begin{proof}[Proof of Theorem~\ref{thm:Lagrange}]
Let us start with an alternative description of the space $L(G)$. Namely, the surface $\hG$ is 
homotopy equivalent to a $4$-valent graph $\Gamma$ on $n$ vertices, obtained from $G$ by 
contracting each ribbon~$E_i$ to the corresponding vertex~$v_i$ (and the rest to its edges); see Fig.~\ref{fig:4-graph}.

\begin{figure}[h!]
\center
\includegraphics{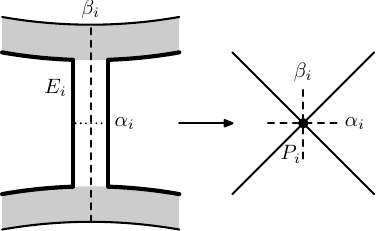}
\caption{Contracting a ribbon graph to a 4-graph}\label{fig:4-graph}
\end{figure}

Then the four half-edges, starting from a vertex $v_i$, correspond to the four vertices of the corresponding ribbon~$E_i$ 
(considered as a rectangle). Finally, these half-edges are joined if the corresponding vertices of the ribbons are joined by a 
part of the boundary of a vertex in the ribbon graph.

The homotopy equivalence gives us an identification between $H_1(\hG;\FF_2)$ and $H_1(\Gamma;\FF_2)$. Now, 
for each ribbon $E_i$ the corresponding summand 
$$
\varphi_{G,i}(\gamma):=(\gamma\cap [\alpha_i]) \cdot e_i + \sum_i (\gamma\cap [\beta_i]) \cdot f_i
$$
in~\eqref{eq:map} under this identification corresponds to a map 
$$
\varphi_{\Gamma,i}:H_1(\Gamma;\FF_2)\to V,
$$
that counts the ``crossing'' by a cycle $\gamma\in H_1(\Gamma;\FF_2)$ of the associated~$\alpha_i$ and~$\beta_i$ (see Fig.~\ref{fig:4-graph}, right).

\begin{figure}[h!]
\center
\includegraphics{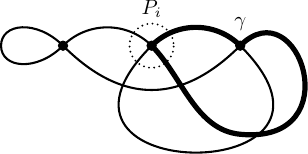}
\caption{The graph $\Gamma$, corresponding to the ribbon graph on the right of Fig.~\ref{fig:chord-base}, the cycle $\gamma$, corresponding to the dashed cycle there, and the neighborhood of the point $P_i$, used to define the restriction map~$R_i$.}\label{fig:restriction}
\end{figure}

For any $i$, denote by $X_i$ the set of half-edges of~$\Gamma$, adjacent to the vertex $P_i$. Let $R_i:H_1(\Gamma;\FF_2)\to 2^{X_i}$ be the restriction map, sending each cycle $\gamma$ to the set of half-edges, adjacent to $P_i$, that it includes (in particular, $R_i$ takes values only in subsets of $X_i$ of even cardinality); {see Fig.~\ref{fig:restriction}}. {Note that as $\Gamma$ is a graph, the cycle~$\gamma$ viewed as a union of edges is well-defined given~$\gamma$ as an element of the homology group}. It is then straightforward to check the following.
\begin{lemma}\label{l:phi-gamma-i}
\begin{enumerate}
\item For any $\gamma\in H_1(\Gamma;\FF_2)$, if $R_i(\gamma)$ consists of zero or four half-edges, one has $\varphi_{\Gamma,i}(\gamma)=0$. Otherwise, $\varphi_{\Gamma,i}(\gamma)\in \{e_i,f_i,e_i+f_i\}$.
\item\label{i:gg} For any $\gamma,\gamma'\in H_1(\Gamma;\FF_2)$, one has $\varphi_{\Gamma,i}(\gamma)=\varphi_{\Gamma,i}(\gamma')\neq 0$ if and only if $R_i(\gamma)$ and $R_i(\gamma')$ consist of two half-edges each, and either coincide, or are complementary in~$X_i$.
\end{enumerate}
\end{lemma}

Applying Lemma~\ref{l:phi-gamma-i}, we get that $\dim \Ker \varphi_{\Gamma}=1$, where $\varphi_{\Gamma}=\sum_i \varphi_{\Gamma,i}$ is the map that corresponds to $\varphi_{G}$ under the identification between $H_1(\hG;\FF_2)$ and $H_1(\Gamma;\FF_2)$. Indeed, if for $\gamma\in H_1(\Gamma;\FF_2)$ one has $\varphi_{\Gamma}(\gamma)=0$, then for any vertex $P_i$ one has $\varphi_{\Gamma,i}(\gamma)=0$, and hence $\gamma$ visits any vertex $P_i$ either by zero, or by all the four adjacent half-edges. An induction argument (together with the connectedness of $\Gamma$) then implies that either $\gamma=0$ or $\gamma=[\Gamma]$. 
As $\dim\Ker \varphi_{\Gamma}=1$, we have $\dim L(G)=\dim \Im \varphi_{\Gamma}=(n+1)-1=n$.

Finally, let us check that $L(G)=\Im \varphi_{\Gamma}$ is a Lagrangian subspace. Indeed, for any two $\gamma,\gamma'\in H_1(\Gamma;\FF_2)$ we have 
$$
(\varphi_{\Gamma}(\gamma),\varphi_{\Gamma}(\gamma')) = \sum_i (\varphi_{\Gamma,i}(\gamma),\varphi_{\Gamma,i}(\gamma')).
$$ 
Now, it is straightforward to check that conclusion~(\ref{i:gg}) in Lemma~\ref{l:phi-gamma-i} can be reformulated in the following way: $(\varphi_{\Gamma,i}(\gamma),\varphi_{\Gamma,i}(\gamma'))=1$ if and only if $R_i(\gamma)\cup R_i(\gamma')$ consists of exactly three half-edges. Indeed, the skew-product of any two different vectors in $\{e_i,f_i,e_i+f_i\}$ is equal to~$1$, and the union of any two different and non-complementary two-element subsets of $X_i$ consists of three elements.

Consider now $\gamma\cup \gamma'$ as a subgraph of~$\Gamma$. Any vertex $P_i$ is its vertex of degree $0$, $2$, $3$ or~$4$. We thus have
\begin{multline*}
(\varphi_{\Gamma}(\gamma),\varphi_{\Gamma}(\gamma')) = \sum_i \begin{cases} 1 & P_i \text{ is of degree } 3 \text{ in } \gamma\cup\gamma' \\ 
0 & \text{otherwise} \end{cases} 
\\
= \# \{P_i \mid P_i \text{ is of odd degree in } \gamma\cup\gamma'\} =0 \quad \text{in } \FF_2,
\end{multline*}
as the number of vertices of odd degree in any graph is even. 
\end{proof}

We conclude the section with the following question, due to E.~Ghys. It is a standard statement that for a compact 3-manifold $M$ with boundary, the kernel of the homomorphism $H_1(\partial M)\to H_1(M)$ is a Lagrangian subspace.
\begin{question}[E.~Ghys]
Is there any relation between Lagrangian subspaces defined by compact $3$-manifolds with boundary and $L$-spaces defined by knots? Can one define the notion of an $L$-space using some naturally constructed $3$-manifold?
\end{question}

\subsection{Action on the $L$-spaces of Vassiliev moves and Morse perestroikas} \label{ssec:vassiliev}

In this subsection we see how do Vassiliev moves on ribbon graphs and partial duals, defined in Subsection~\ref{ssec:partialdual}, act on the corresponding $L$-spaces. It turns out that they are nice base changes. Namely, we have the following
\begin{theorem}\label{thm:lspaceformulas}
Let $G$ be a ribbon graph. Then,
\begin{itemize}
\item For any chord $i$, we have $L(\mu_i(G))=M_i(L(G))$, where $M_i:V\to V$ interchanges $e_i$ and $f_i$ and fixes all the other elements of the standard base of~$V$.
\item For any two chords $i,j$ we have $L(v_1^{ij}(G))=T_1^{ij}(L(G))$, where 
$$
T_1^{ij}(f_i)=f_i+e_j,  \quad T_1^{ij}(f_j)=f_j+e_i, \quad T_1^{ij}(e_i)=e_i, \quad T_1^{ij}(e_j)=e_j,
$$
and $T_1^{ij}$ fixes all the other standard base elements of~$V$.
\item For any two chords $i,j$ we have $L(v_2^{ij}(G))=T_2^{ij}(L(G))$, where 
$$
T_2^{ij}(e_i)=e_i+e_j,  \quad T_2^{ij}(f_j)=f_j+f_i, \quad T_2^{ij}(f_i)=f_i, \quad T_2^{ij}(e_j)=e_j,
$$
and $T_2^{ij}$ fixes all the other standard base elements of~$V$.
\end{itemize}
\end{theorem}

Note, that a priori we cannot say that such maps $T_1^{ij}$, $T_2^{ij}$ exist: in the same way as two different chord diagrams may have the same intersection matrix, two different ribbon graphs may have the same $L$-space, and it is not evident that after a Vassiliev move their $L$-spaces will still coincide.

\begin{proof}
We start with the Morse perestroika (or partial dual) operation. Namely, note that for two ribbon graphs, $G$ and $G'=\mu_i(G)$, there is a natural identification between $H_1(\hG;\FF_2)$ and $H_1(\widehat{G'};\FF_2)$. Moreover, the corresponding 4-valent graphs, introduced in the proof of Theorem~\ref{thm:Lagrange}, almost coincide; the only difference between them is that the dashed lines $\alpha_i$ and $\beta_i$  are interchanged (while all the other $\alpha_j$, $\beta_j$ stay unchanged). Thus, for any $\gamma\in H_1(\hG;\FF_2)$ its images (after the identification mentioned above) under $\varphi_{\widehat{G}}$ and $\varphi_{\widehat{G'}}$ differ by the application of a linear map $M_i:V\to V$ that interchanges $e_i$ and $f_i$ and fixes all the other base elements $e_j$ and $f_j$. But this immediately implies that 
$$
L(G')=\Im \varphi_{G'} = \Im M_i\circ \varphi_G = M_i(\Im \varphi_G)= M_i(L(G)).
$$

Next, consider the first Vassiliev move $v_1^{i j}$ for the $i$-th and the $j$-th chord. Again, there is a natural identification between $H_1(\widehat G,\FF_2)$ and $H_1(\widehat G',\FF_2)$, where $G'=v_1^{ij}(G)$. Indeed, for any $\gamma\in H_1(\widehat G,\FF_2)$ we do not change it outside the area where the Vassiliev move is applied (Fig.~\ref{fig:ribbon_moves}, left), and we let it ``go along the same ribbons'' (and consequently close up) inside it. Then, it is easy to see that for any $\gamma\in H_1(G,\FF_2)$ for the corresponding $\gamma'$ almost all the corresponding intersections stay unchanged:
\begin{equation}\label{eq:1-unchanged}
\forall k\neq i,j \quad [\alpha_k]\cap \gamma=[\alpha_k']\cap \gamma', \quad [\beta_k]\cap \gamma=[\beta_k '] \cap \gamma'.
\end{equation}
Moreover, the intersections with $\alpha_i, \alpha_j$ also stay unchanged (the cycles pass through the same ribbons). So, the only difference will be with the intersection with $\beta_i$, $\beta_j$. For such intersections, one has 
\begin{equation}\label{eq:beta-i}
\gamma' \cap [\beta_i'] = \gamma \cap [\beta_i] + \gamma\cap [\alpha_j],
\end{equation}
\begin{equation}\label{eq:beta-j}
\gamma' \cap [\beta_j'] = \gamma \cap [\beta_j] + \gamma\cap [\alpha_i],
\end{equation}
as interchanging of the endpoints of $i$th and $j$th ribbon has changed the intersection index with $[\beta_i]$ for cycles that pass along $j$th ribbon, and with  $[\beta_j]$ for cycles that pass along $i$th ribbon. Substituting~\eqref{eq:1-unchanged}--\eqref{eq:beta-j} in~\eqref{eq:map}, we see that 
$$
\varphi_G(\gamma)=T_1^{ij} (\varphi_{G'}(\gamma')),
$$
where $T_1^{ij}$ fixes all the basis elements of $V$, except for $f_i$ and $f_j$, and 
$$
T_1^{ij}(f_i)=f_i+e_j,  \quad T_1^{ij}(f_j)=f_j+e_i.
$$
We thus get the desired equality
$$
L(G') = T_1^{ij} (L(G)). 
$$

Finally, the last conclusion is implied by the first and the second one: from Proposition~\ref{prop:conjugacy} we know that $v_2^{ij} = \mu_i v_1^{ij} \mu_i$, and we have 
$$
T_2^{ij} = M_i T_1^{ij} M_i.
$$
\end{proof}

\section{Bialgebra structures}\label{sec:bialgebra}

A well-known feature of chord diagrams is that they generate a bialgebra. In this section we first recall this construction and then provide its generalizations: the \emph{graph bialgebra}, due to Lando~\cite{Lando00}, and the \emph{bialgebra of $L$-spaces}, which is the main construction of this paper. 

\subsection{The bialgebra of chord diagrams}\label{ssec:chord}

Let $\KK$ be a field (or, even more generally, a commutative and associative ring). For $n>0$, let $\cA_n$ be a $\KK$-vector space formally spanned by all chord diagrams with $n$ chords; we set $\cA_0=\KK$ and $\cA=\bigoplus_{n\geq 0}\cA_n$.

For any chord diagram $c$ and any pair of neighboring chords defining Vassiliev moves, the alternating sum
\[
 c-v_1(c)-v_2(c)+v_1v_2(c)
\]
is called a \emph{four-term element}.

Let $\cA^{(4)}$ be the subspace of $\cA$ generated by all four-term elements. Denote by $\cM$ the quotient space: $\cM=\cA/\cA^{(4)}$. This graded vector space can be turned into a graded bialgebra by introducing the operations of multiplication and comultiplication.

To define the comultiplication on $\cM$, let us first define it on $\cA$. Take a chord diagram $c$ with $n$ chords. Let $V(c)$ be the set of its chords; for a subset $I\subset V(c)$ denote by $c_I$ the diagram formed by all chords from the set $I$. Then the comultiplication $\overline\Delta\colon \cA\to\cA\otimes\cA$ takes $c$ into
\[
 \overline\Delta(c)=\sum_{I\sqcup I' = V(c)} c_I\otimes c_{I'}.
\]
A routine check shows that $\overline\Delta(\cA^{(4)})\subset \cA^{(4)}\otimes \cA + \cA\otimes \cA^{(4)}$, so this operation determines a comultiplication $\Delta\colon\cM\to\cM\otimes\cM$ on the quotient space.

Multiplication of two chord diagrams is defined as their ``connected sum''. To multiply two chord diagrams $c_1$ and $c_2$, we make a puncture in each of the two circles and attach these two circles one to another along this puncture, obtaining a new diagram $c_1\# c_2$. Clearly, this operation is not well-defined, since the result depends on the positions of the punctures. However, it is not hard to show that all such chord diagrams are congruent \emph{modulo the four-term relations}. This allows us to define the product on $\cM$ as follows: $c_1\cdot c_2=c_1\# c_2\mod\cA^{(4)}$. Clearly, it is commutative.

One can also show that these two operations satisfy the axioms of a bialgebra, thus obtaining the following theorem (cf.~\cite[Theorem~6.1.12]{LZ05}).
\begin{theorem}
 The vector space of chord diagrams $\cM$ is a graded commutative and cocommutative bialgebra over $\KK$ with respect to the operations introduced above.
\end{theorem}

\subsection{The bialgebra of graphs}\label{ssec:bialg_graphs}

In this subsection we define the bialgebra of graphs. This is a purely combinatorial anagolue of the bialgebra of chord diagrams. It was introduced by Lando in \cite{Lando00}. We mostly follow this paper and \cite[Sec.~6.4]{LZ05}.

Let $\cG_n$ be a graded $\KK$-vector space freely spanned by all graphs (not necessarily connected) on $n$ vertices, and let $\cG=\bigoplus_{n\geq 0}\cG_n$. This vector space admits a structure of a bialgebra in the following way. The multiplication $m\colon\cG_k\otimes\cG_n\to\cG_{k+n}$ brings a pair of graphs into their disjoint union. The unit of this multiplication is represented by the empty graph.

For a graph $\Gamma$ and a subset $J\subset V(\Gamma)$ of its vertices, denote by $\Gamma_J$ the \emph{restriction} of $\Gamma$ to the set $J$: a subgraph of $\Gamma$ formed by the vertices from $J$ and the edges of $\Gamma$ such that both their ends belong to $J$. Then the comultiplication $\Delta\colon\cG\to\cG\otimes\cG$ is defined as follows:
\[
 \Delta(\Gamma)=\sum_{J\sqcup J'=V(\Gamma)} \Gamma_J\otimes\Gamma_{J'}.
\]
It resembles the comultiplication in the bialgebra of chord diagrams. We will see that there is a close relation between these two bialgebras.

In Subsection~\ref{ssec:IntGraph} we have assigned to each chord diagram $c$ of order $n$ its intersection graph $\Gamma(c)$. Thus we have obtained a map $\iota\colon\cA\mapsto\cG$. As we have already discussed, this map is neither injective nor surjective.
However, the notion of Vassiliev moves and 4-elements can be extended to $\cG$. This is done as follows. Let $A$ and $B$ be two distinct vertices of a graph $\Gamma$. Define the \emph{first Vassiliev move} $v_1(\Gamma)$ as the graph obtained from $\Gamma$ by removing the edge $AB$ if this edge exists or by adding this edge otherwise. The definition of the \emph{second Vassiliev move} $v_2(\Gamma)$ is defined as follows: for each vertex $C\in V(\Gamma)\setminus\{A,B\}$, we change its adjacency with $A$ if $C$ is joined with $B$, and do nothing otherwise. All other edges in this graph remain the same. Note that $v_2(\Gamma)$ depends on the order of the vertices $(A,B)$. Clearly, $v_1$ and $v_2$ commute.

\begin{remark} In the original paper \cite{Lando00} $v_1(\Gamma)$ and $v_2(\Gamma)$ are denoted by $\Gamma'$ and $\widetilde\Gamma$, respectively.
\end{remark}

As in the case of chord diagrams, we define a \emph{4-element} as
\[
 \Gamma-v_1(\Gamma)-v_2(\Gamma)+v_1v_2(\Gamma).
\]
All 4-elements span a subspace in $\cG$ denoted by $\cG^{(4)}$. We also denote the quotient of $\cG$ modulo the subspace of 4-elements by $\cF$:
\[
 \cF=\cG/\cG^{(4)}.
\]
 $\cF$ is called the \emph{4-bialgebra of graphs}. One can check that 4-elements are compatible with the multiplication and comultiplication, so the following theorem holds.

\begin{theorem} $\cF$ is a commutative and cocommutative bialgebra with respect to the multiplication and comultiplication described above. 
\end{theorem}

Clearly, $\iota(\cA^{(4)})\subset\cG^{(4)}$. So we get a well-defined map of bialgebras
\[
 \overline\iota\colon \cM\to\cF.
\]
This map is known to be non-injective: the injectivity does not hold for the 7-th graded component. Conjecturally, it is surjective (recall that $\iota$ is not, what makes this conjecture nontrivial).

Let us point out two fundamental differences between chord diagrams and graphs. First, the Vassiliev moves for a chord diagram $c\in\cA$ can only be defined for an ordered pair of \emph{neighboring} chords, while for the case of graphs this definition makes sense for an arbitrary ordered pair of vertices. Second, the span of the set of graphs $\cG$ forms a bialgebra even before the factorization over the 4-term relations, while the span of the chord diagrams $\cA$ does not admit a well-defined multiplication. This makes the 4-bialgebra of graphs in a sense ``nicer'' than the bialgebra of chord diagrams.

\subsection{The bialgebra of $L$-spaces}\label{ssec:bialgebra_lspaces}

A natural idea would be to extend the bialgebra structure on the space spanned by chord diagrams (or, equivalently, orientable one-vertex ribbon graphs) to the space of \emph{all} ribbon graphs. Unfortunately, this fails even on the stage of framed chord diagrams (one-vertex ribbon graphs, not necessarily orientable), modulo the four-term relation. As it was discussed in \cite{Lando06} and shown in \cite{IlyutkoManturov15}, the attempt to multiply framed chord diagrams in a similar way as in the previous subsection fails: this multiplication is not well-defined. (Recently M.~Karev \cite{Karev14} showed that framed chord diagrams form a module over the bialgebra of chord diagrams.) 

However, the $L$-spaces of ribbon graphs do have a bialgebra structure. We have seen earlier that $L$-spaces can be viewed as generalizations of intersection matrices/graphs, so this structure is naturally extended from the graph bialgebra. Just as in the case of graphs, the multiplication and comultiplication of $L$-spaces is well-defined even without the four-term relation.

We have seen in Example~\ref{ex:int} that $L$-spaces of ribbon graphs generalize the notion of intersection matrices for chord diagrams. Intersection matrices are defined by their adjacency graphs; the only difference is that they depend on a specific \emph{ordering} of chords, while the vertices in graphs are not ordered. To define a proper analogue of an intersection matrix, we will consider not just $L$-spaces, but rather their \emph{orbits} under the action of a symmetric group.

Let $\FF_2^{2n}=\langle e_1,f_1,\dots,e_n,f_n\rangle$ be a $2n$-dimensional vector space over $\FF_2$ with a standard skew-symmetric form defined by $(e_i,f_j)=\delta_{ij}$. Consider a Lagrangian Grassmannian $LGr(n)$: this is the set of all maximal (i.e., $n$-dimensional) isotropic subspaces in $\FF_2^{2n}$. We will consider it just as a finite set, without using any additional structure on it. The symmetric group $S_n$ acts on $V$ by simultaneous permutations of~$e_i$'s and~$f_i$'s. This action preserves the symplectic form and yields an action of $S_n$ on $LGr(n)$. Consider the $\KK$-vector space spanned by the set of orbits $LGr(n)/S_n$ of the latter action; denote it by $\cL_n$. Let
\[
 \cL=\bigoplus_{n\geq 0}\cL_n,
\]
where $\cL_0=\KK$. 

Let $\Gamma\in\cG_n$ be a graph. Pick an arbitrary numbering of its vertices and take the adjacency matrix $M(\Gamma)$. Consider an $n\times 2n$-matrix ${(\Id_{n}\mid M(\Gamma))}$. Its rows span a Lagrangian subspace $L(\Gamma)\subset\FF_2^{2n}$. This gives us an embedding
\[
\cG\to\cL.
\]
of the bialgebra of graphs into $\cL$. Our next goal is to show that $\cL$ has a bialgebra structure compatible with the multiplication and comultiplication on $\cG$. To do this, let us define the operations on $\cL$.

The multiplication on is defined as follows. Let $L_1\subset \FF_2^{2n}$ and $L_2\in \FF_2^{2m}$ be two Lagrangian subspaces; then their direct sum $L_1\oplus L_2\subset \FF_2^{2(m+n)}=\FF_2^{2n}\oplus \FF_2^{2m}$ is also Lagrangian. Then the multiplication map
\[
\cL_n\otimes\cL_m\to \cL_{m+n} 
\]
is defined by
\[
 [L_1]\cdot[L_2]\mapsto [L_1\oplus L_2],
\]
where the square brackets stand for orbits of the symmetric groups on the corresponding Lagrangian Grassmannians. Clearly, this product is well-defined, commutative, associative, has a unit $\{0\}\subset \FF_2^0$, and its restriction to $\cG$ gives the multiplication on $\cG$ obtained by taking the disjoint union of graphs.

To define comultiplication in $\cL$, we need to generalize the operation of restricting a graph to a subset of its vertices. This is done by \emph{symplectic reduction}.

Namely, for any $I\subset \mN=\{1,\dots,n\}$ let 
$$
E_I:=\langle \{e_i \mid i\in I \}\rangle, \quad F_I:=\langle \{f_i \mid i\in I \}\rangle.
$$ 
The space $W_I:=E_I\oplus F_{\mN}$ is then coisotropic: it contains its $(\cdot,\cdot)$-orthogonal complement $W_I^{\perp}=F_J$, where $J=\mN\setminus I$.

Associating to any Lagrangian subspace $L\subset \FF_2^{2n}=E_{\mN}\oplus F_{\mN}$ the space $L|_{I} := \pi_I (L\cap W_I)$, where $\pi_I:W_I\to W_I/W_I^{\perp} = E_I\oplus F_I$ is the natural projection map,  we obtain the desired generalization of the restriction operation. Indeed, for any framed graph $\Gamma$ with the adjacency matrix $M=\left(\begin{smallmatrix} M_{I,I} & M_{I,J} \\ M_{J,I} & M_{J,J}
\end{smallmatrix}\right)$ the associated $L$-space $L(\Gamma)$ is generated by rows of a matrix
\begin{equation}\label{eq:L(D)}
{\left(\left.
\begin{smallmatrix} \Id_I & 0 \\ 0 & \Id_J
\end{smallmatrix} 
\right| 
\begin{smallmatrix} M_{I,I} & M_{I,J} \\ M_{J,I} & M_{J,J}
\end{smallmatrix} 
\right) }.
\end{equation}
Hence, the intersection with $W_I$ maps $L(\Gamma)$ into the space generated by rows of a matrix 
$$
{(\Id_I \, 0 \mid M_{I,I} \, M_{I,J}),}
$$
which is sent by the projection $\pi_I$ to ${(\Id_I \mid M_{I,I})}$, which is exactly the $L$-space~$L(\Gamma|_{I})$ of the graph $\Gamma$ restricted to the subset $I$ of its vertices.

Finally, a standard lemma says that $L|_{I}$ is always a Lagrangian subspace:
\begin{lemma}[Symplectic reduction lemma]\label{l:s-red}
For any Lagrangian subspace $L\subset \FF_2^{2|\mN|}$ and any $I\subset \mN$, the space $L|_{I}\subset \FF_2^{2|I|}$ is also Lagrangian.
\end{lemma}
\begin{proof}
The bilinear form $(\cdot,\cdot)$ restricted on~$L$ vanishes. Hence, the same is true for its restriction on $L\cap W_I$, and thus on $\pi_I(L\cap W_I)=L|_{I}$. This means that $L|_{I}$ is isotropic. 

On the other hand,
$$
\dim L|_{I} = \dim (L\cap W_I) - \dim ((L\cap W_I)\cap W_I^{\perp}) = \dim (L\cap W_I) - \dim (L\cap W_I^{\perp}).
$$
As $L$ is Lagrangian, $L=L^{\perp}$ and hence $\dim (L\cap W_I^{\perp})=\dim (L^{\perp}\cap W_I^{\perp})=\dim ((L+ W_I)^{\perp})$. Thus,
\begin{multline*}
\dim L|_{I} = \dim (L\cap W_I) - \dim ((L+ W_I)^{\perp}) =  \dim (L\cap W_I) - \\ - 2 \dim L + \dim (L+ W_I) = \dim L + \dim W_I - 2 \dim L = |I|.
\end{multline*}
Hence, $L|_{I}\subset \FF_2^{2|I|}$ is Lagrangian.
\end{proof}

Now, define the comultiplication on the space of Lagrangian subspaces as 
$$
\Delta(L):=\sum_{I\subset \mN} L|_I\otimes L|_{\mN\setminus I}.
$$

It is easy to see that for any $I\subset I'\subset \mN$ one has $L|_{I}=(L|_{I'})|_{I}$. This implies coassociativity:
$$
((\Id\otimes \Delta) \circ \Delta)(L) =  ((\Delta\otimes \Id) \circ \Delta) (L) = \sum_{I_1\sqcup I_2\sqcup I_3=\mN} L|_{I_1}\otimes L|_{I_2}\otimes L|_{I_3}.
$$

Summarizing, we obtain the following theorem.
\begin{theorem} $\cL$ is a commutative, cocommutative, associative and coassociative bialgebra with respect to the operations described above. The canonical embedding $\cG\hookrightarrow\cL$ of the bialgebra of graphs is a homomorphism of bialgebras.
\end{theorem}

\subsection{The four-bialgebra of $L$-spaces}\label{ssec:bialgebra_4lspaces}

Similarly to the cases of chord diagrams and graphs, we can introduce the notion of four-elements in the bialgebra $\cL$ and consider the quotient of $\cL$ modulo the ideal of four-elements.

Namely, let $L\subset\FF_2^{2n}$ be a Lagrangian subspace, regarded as an element of $\cL_n$. As in all previous cases (for chord diagrams, non-framed and framed graphs), we can define Vassiliev moves acting on $L$ as in Subsection~\ref{ssec:vassiliev}. Let $v_1$ and $v_2$ be two symplectomorphisms on $\FF_2^{2n}$, defined as in Theorem~\ref{thm:lspaceformulas}. The images $v_1(L)$ and $v_2(L)$ of $L$ under these transforms can be again regarded as elements of $\cL_n$. Then
\[
 L-v_1(L)-v_2(L)+v_1v_2(L)
\]
is said to be a \emph{four-element}. As before, we denote by $\cL^{(4)}$ the homogeneous ideal in $\cL$ generated by all 4-elements. The quotient of $\cL$ modulo this ideal is denoted by $\cK$. It also inherits a natural grading from $\cL$: 
\[
 \cK=\bigoplus_{n\geq 0}\cK_n.
\]
\begin{theorem}
 The multiplication and comultiplication defined above turn $\cK$ into a commutative and cocommutative bialgebra. 
\end{theorem}

\begin{proof}
The only thing we need to check is that the multiplication and comultiplication respect the four-term relation. For multiplication this is obvious, while for comultiplication the proof repeats the corresponding proof for graphs, see~\cite[Theorem~2.4]{Lando00}.
\end{proof}

\subsection{Some remarks and open questions}

For the bialgebras of chord diagrams and graphs one can be interested in the sequences of dimensions of graded components of their 4-bialgebras $\dim\cM_n$ and $\dim\cF_n$. These sequences are quite mysterious; closed formulas or generating functions for them are unknown. The computations were carried out by Vassiliev, Bar-Natan, Kneissler, Lando and Soboleva; see \cite[Sec.~6,1,~6.4]{LZ05} for an overview. It would be interesting to look at the beginning of the corresponding sequence for the four-bialgebra of $L$-spaces. So the first problem is as follows.

\begin{problem} Compute $\dim\cK_n$.
 \end{problem}

Another question returns us to the definition of $L$-spaces for ribbon graphs:

\begin{problem} Are all elements of $\cK_n$ obtained as images of linear combinations of actual ribbon graphs with $n$ edges? If not, how close is this map to surjection, i.e., what is the codimension of its image?
\end{problem}

In Sec.~\ref{ssec:def_lspace} we have considered the bialgebra of $L$-spaces, generated by Lagrangian subspaces in  a symplectic vector space over $\FF_2$. However, we can consider the set of Lagrangian subspaces over an arbitrary field of finite or infinite characteristic or even over $\ZZ$. If the ground field is finite, we can span a bialgebra by this set, just as in Sec.~\ref{ssec:chord}. The formulas from 
Theorem~\ref{thm:lspaceformulas} define Vassiliev moves on Lagrangian subspaces; these moves are not involutive anymore, their order is equal to the characteristic of the ground field. An easy check shows that these moves still commute. Hence our last question:

\begin{problem}
What is the right analogue of four-term relations for the bialgebra of Lagrangian subspaces over $\FF_p$ or $\ZZ$? What bialgebra is obtained after factorization over these relations?
\end{problem}

A close object to those studied in this paper is the one of delta-matroids: one can associate a delta-matroid to any ribbon graph~\cite{Moffatt2014}. It looks interesting to investigate this link further; in particular, after this paper was finished, we were informed of the paper~\cite{Oum2012}, where Lagrangian subspaces are considered in relation with delta-matroids, as well as of the text in preparation by Lando and Zhukov~\cite{LandoZhukov2015}, where a graded bialgebra, related to delta-matroids, is introduced.

\section{Miscellaneous}\label{sec:misc}

\subsection{Action of Morse perestroikas in the one-component case}

Consider now the particular case of one-vertex ribbon graphs and their intersection matrices. A question that naturally arises, is when a perestroika maps a one-vertex diagram to a one-vertex graph. The following proposition, first obtained (in slightly different terms, see Remark~\ref{r:bridges} below) by Cohn and Lempel~\cite[Thm. 1]{CL1972}, answers it:

\begin{proposition}\label{p:M-pi}
Let $G$ be a one-vertex ribbon graph (or, what is the same, a framed chord diagram), and let $J\subset \mN$ be a set of indices. The image $\mu_{J}(G)$ is a one-vertex ribbon graph if and only if the minor $\det H$ of the intersection matrix $M=\left(
\begin{smallmatrix} A & B \\
B^* & H
\end{smallmatrix} 
\right)$, corresponding to the set $J$ of indices, is non-zero.
\end{proposition}

\begin{proof}
Note first that a ribbon graph $G'$ has only one vertex if and only if the corresponding $L$-space $L(G')$ is transversal to the subspace~${F}_\mN$. In one direction it can be seen immediately out of~\eqref{eq:L(D)}, in the other one, one can easily see that the image of any component under the map~$\varphi$ from Def.~\ref{def:lspace} belongs to~${F}_{\mN}$.

Applying now Theorem~\ref{thm:lspaceformulas} to~\eqref{eq:L(D)}, we see that the $L$-space $L(\mu_J(G))$ can be generated by rows of the matrix
\begin{equation}\label{eq:pi}
{\left(\left.
\begin{smallmatrix} \Id_I & B \\ 0 & H
\end{smallmatrix} 
\right|
\begin{smallmatrix} A & 0 \\ B^{{*}} & \Id_J
\end{smallmatrix} 
\right)}.
\end{equation}
Now, $L(\mu_J(G))$ is transverse to ${F}_\mN$ if and only if the right half of the matrix~\eqref{eq:pi} is non-degenerate, what is in its turn equivalent to the non-degeneracy of the submatrix~$H$.
\end{proof}

\begin{remark}\label{r:bridges}
In a slightly different language, Prop.~\ref{p:M-pi} is discussed also in~\cite[Remark 6.4.17]{LZ05} and \cite[4.8.6]{CDM} (see also Moran~\cite{Moran}). 
Namely, given a one-component chord diagram, one can thicken the boundary circle, and replace the chords corresponding to the subset $J$ of indices by ``bridges'' (see Fig.~\ref{f:AB}). Then,~\cite[Remark 6.4.17]{LZ05} states that the minor of the intersection matrix corresponding to these chords is non-zero if and only if the ``interior'' boundary of the obtained figure is also a circle. To conclude the construction of $\mu_J(C)$, it suffices now to shift the non-thickened chords endpoints from the initial boundary circle to the interior one, and replace the thickened chords by ``crossings'' of the corresponding bridges. 
\end{remark}

Moreover, if the matrix $H$ in Prop.~\ref{p:M-pi} is non-degenerate, we can write an explicit formula for the new intersection matrix:
\begin{proposition}\label{p:ABH}
Let $M$ be the intersection matrix of a one-vertex ribbon graph~$G$ with an intersection matrix $M=\left(
\begin{smallmatrix} A & B \\
B^* & H
\end{smallmatrix} 
\right)$, and the minor $\det H$ is non-zero. Then, the intersection matrix of the one-component chord diagram $\mu_J(G)$ is 
\begin{equation}\label{eq:gen-pivot}
\left(
\begin{matrix} 
A +  BH^{-1}B^* & BH^{-1} \\
H^{-1}B^* & H^{-1}
\end{matrix} 
\right)
\end{equation}
\end{proposition}
\begin{proof}
It suffices to pass from the basis~\eqref{eq:pi} to a basis of the form~\eqref{eq:L(D)}. To do so, we multiply the rows, corresponding to the set $J$ of indices, by {$H^{-1}$}, and add the resulting rows to those corresponding to $I$ ones with a multiplication by~$B$, obtaining the matrix
$$
{\left(\left.
\begin{matrix} 
\Id_{I} & 0\\
0 & \Id_{J}
\end{matrix}
\, \right|\,
\begin{matrix} 
A +  BH^{-1}B^* &  BH^{-1} \\
H^{-1}B^* & H^{-1}
\end{matrix} 
\right)}
$$

\end{proof}

\begin{example}\label{ex:ab}
Two simplest examples when Prop.~\ref{p:ABH} is applicable are perestroikas with respect to a single disorienting chord and with respect to two intersecting orientable chords.  
The explicit formula~\eqref{eq:gen-pivot} in these cases respectively becomes:
$$
\left(
\begin{matrix} 
M' & v_a \\
v_a^* & 1
\end{matrix} 
\right) \mapsto \left(
\begin{matrix}  M'+v_av_a^* & v_a \\
v_a^* & 1
\end{matrix} 
\right)
$$
and 
$$
\left(
\begin{matrix} 
M' & v_a & v_b \\
v_a^* & 0 & 1 \\
v_b^* & 1 & 0
\end{matrix} 
\right) \mapsto \left(
\begin{matrix}  
M'+v_av_b^* + v_bv_a^* & v_b & v_a \\
v_b^* & 0 & 1 \\
v_a^* & 1 & 0
\end{matrix} 
\right)
$$
The last formula can also be found in \cite[Lemma 1]{CL1972}.
\end{example}

As we will see in the next section, they are related to the \emph{local complementation} and \emph{pivot} graph operations, respectively.

\begin{figure}
\center
\center
\includegraphics{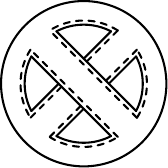}
\caption{Two thickened intersecting chords and the interior circle}\label{f:AB}
\end{figure}

\subsection{Pivot and local complementation on graphs} In \cite{ABS02a} and \cite{ABS02b} Arratia, Bollob\'as and Sorkin define the \emph{interlace polynomial} of a framed graph. This graph invariant (in its two-variable version) is defined as follows:
\[
q(G,x,y)=\sum_{S\subset V(G)} (x-1)^{r(G[S])} (y-1)^{n(G[S])},
\]
where $G$ is a graph, $V(G)$ is its set of vertices, $S$ in the sum runs over all $2^{|V(G)|}$ subsets of $V(G)$, and $r(G[S])$ and $n(G[S])$ stand for the rank and the nullity of the adjacency matrix of $G$ restricted to the set of vertices~$S$. This definition is somewhat similar to the definition of the Tutte polynomial of a graph.

Later N.~Netrusova \cite{Netrusova} proved that for non-framed graphs the interlace polynomial is a 4-invariant and, as a corollary, defines a knot invariant.

The interlace polynomial of a graph can be computed inductively, by reducing the computation to graphs with the smaller number of vertices. These reduction formulas use the operations of \emph{local complementation} $G\mapsto G^a$ with respect to a vertex $a$, provided that $a$ is odd, and \emph{pivot} $G\mapsto G^{ab}$ with respect to an edge $ab$ of a graph, provided that both $a$ and $b$ are even. See \cite[Thm~3, Thm~6]{ABS02b} for details.

These two operations have a very simple interpretation in the language of $L$-spaces. Consider the $L$-space $L(G)$ of a graph $G$: if $A=A(G)$ is the adjacency matrix of $G$, then $L(G)$ is spanned by the rows of the matrix $(\Id_n|A)$.

Then taking the local complementation with respect to a vertex $a$ corresponds to the Morse perestroika $\mu_a(L(G))$. Note that since $a$ is odd, then $\mu_a(L(G))$ still intersects $E_\mN$ transversely, i.e., can be presented in the form $(\Id_n|\mu_a(A))$. Comparing the formulas from \cite[Lemma~5]{ABS02b} and from Example~\ref{ex:ab}, one can see that the matrix $\mu_a(A)$ is nothing but the adjacency matrix of the graph $G^a$ obtained from $G$ by the local complementation in $a$. In other words, the local complementation is just a Morse perestroika in an odd vertex.

Similarly one can see that the pivot with respect to an edge $ab$ joining two even vertices is the composition of two Morse perestroikas $\mu_a\mu_b$ followed by the change of labels on these two vertices: the vertex $a$ becomes $b$, and vice versa.

\subsection{Remarks on plane curves}

The original motivation for defining the bialgebra of chord diagrams was the study of Vassiliev knot invariants. Chord diagrams correspond to singular knots, i.e., knots with finitely many transverse simple self-intersections (these are the elements of the main stratum of the discriminant set for the space of knots). The four-term relation then comes from different ways to resolve a triple intersection in singular knots, and the finite-type knot invariants are linear functions on the bialgebra of chord diagrams. Similarly, the graph bialgebra can also be used for constructing such invariants (\cite{CDM}, \cite{LZ05}).

There is another topological object similar to knots: plane curves. The study of their invariants began with V.~Arnold's paper \cite{Arnold93}. These are oriented real  plane curves, possibly with finitely many transversal self-intersections. Their invariants can be studied in a way similar to Vassiliev knot invariants; the main stratum of the discriminant for them consists of plane curves with finitely many \emph{self-tangencies} of quadratic type. Each self-tangency can be either direct, where the velocity vectors of the curve at the self-tangency point have the same direction, or inverse, where they point in different directions. Thus, a plane curve gives us a \emph{framed} chord diagram, with even/odd framing corresponding to direct and inverse self-tangencies, respectively. For them one can carry out more or less the same program as in the case of knots, defining Vassiliev moves, four-term relation etc.; see \cite{Goryunov98} for details. 
%However, as it was discussed in~\cite{Lando06}, the {framed} chord diagrams modulo these relations are not known to have a bialgebra structure. 
{However (see the discussion in~\cite{Lando06}), for a certain time it was supposed that the {framed} 
chord diagrams modulo these relations do not form a bialgebra. And indeed, it was shown recently by 
Ilyutko and Manturov in~\cite{IlyutkoManturov15}.}

\begin{figure}[h!]
\center
\includegraphics[scale=0.5]{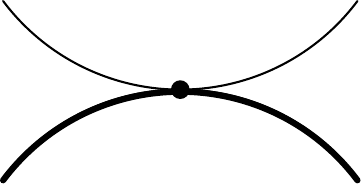} \qquad 
\includegraphics[scale=0.5]{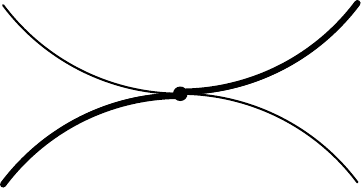}
\caption{A conjectural action of Morse perestroika on plane curves}\label{fig:curve_perestroika}
\end{figure}

In a  similar way each \emph{multicomponent} plane curve with self-tangencies gives rise to a ribbon graph. However, a problem arises while trying to define Morse perestroikas in such a way that they would bring a smooth multicomponent plane curve into another smooth multicomponent plane curve from the main stratum (i.e., with quadratic self-tangencies). We would expect a Morse perestroika to switch the branches in a neighborhood of a self-tangency point, as shown on Fig.~\ref{fig:curve_perestroika}, but this operation would bring a quadratic self-tangency into a cubic one, and the latter curve is not in the main discriminant stratum.

\section*{Acknowledgements}
This work grew out of a question by S.~Lando about  generalization of the notion of interlace polynomials to framed graphs. We would like to thank S.~Lando, M.~Kazaryan and V.~Vassiliev for useful discussions and comments. We also express our gratitude to N.~Netrusova and A.~Vorontsov; we benefited a lot from discussions with them. Finally, we would like to thank the referee for his careful reading, useful references and helpful remarks.

The article was prepared within the framework of a subsidy granted to the HSE by the Government of the Russian Federation for the implementation of the Global Competitiveness Program. The authors were partially supported by RFBR projects 13-01-00969-a and 16-01-00748-a (V.K.), Dynasty foundation and Simons-IUM fellowship (E.S.). This work started during E.S.'s visit to the University of Rennes~1; we are grateful to this institution for excellent working conditions and to CNRS for the financial support of the visit.

We dedicate this paper to the memory of our dear friend and teacher Sergei Vassilievich Duzhin (1956--2015).

\bibliographystyle{plain}
\bibliography{pivots}

\end{document}